\documentclass[12pt,a4paper,oneside]{amsart}
\usepackage[a4paper]{geometry}
\usepackage{mathptmx} 

\usepackage{amssymb,enumerate,mathrsfs,tikz-cd}

\newcommand{\vertbar}{\>|\>}
\newcommand{\set}[2]{\ensuremath{\{ #1 \vertbar #2 \}}}
\newcommand{\modulus}[1]{\ensuremath{|\, #1 \,|}}

\DeclareMathOperator{\Ker}{Ker}

\newtheorem*{theorem}{Theorem}
\newtheorem{proposition}{Proposition}
\newtheorem*{lemma}{Lemma}

\begin{document}

\title{On the utility of Robinson--Amitsur ultrafilters. III}
\author{Pasha Zusmanovich}
\address{University of Ostrava, Czech Republic}
\email{pasha.zusmanovich@osu.cz}
\date{First written August 6, 2020; last minor revision November 12, 2022}
\thanks{Periodica Mathematica Hungarica, to appear}

\begin{abstract}
We give a new proof of a version of the main theorem of the previous paper in 
the series about embedding of an algebraic system into ultraproducts.
\end{abstract}

\maketitle

\subsection{Introduction}

This note is a postscript to the earlier paper \cite{ultra}. Theorem 1 in that
paper asserts that, given a cardinal $\kappa$, an embedding of a 
$\kappa$-subdirectly irreducible algebraic system into a direct product can be 
factored to get an embedding into an ultraproduct over a $\kappa$-complete 
ultrafilter. However, the proof needed an assumption that $\kappa$ is a strongly
compact cardinal (i.e., that any $\kappa$-complete filter can be extended to a 
$\kappa$-complete ultrafilter), an assumption which, as noted in 
\cite[\S 3]{ultra}, brings us to deep waters of set theory. So it seems to be
natural to try to get rid of this assumption. This is what we are doing in this
note, at the expense of somewhat narrowing the scope: in addition to 
$\kappa$-subdirect irreducibility, we require a sort of a stronger version of
subdirect irreducibility, dubbed by us ``indecomposability'' (for an exact 
definition of this notion and its properties, see \S \ref{sec-decomp}; we 
emphasize that this is a version of the finite subdirect irreducibility, and not
of $\kappa$-subdirect irreducibility).

Another interesting feature of the proof presented here is the following. 
Another theorem proved in \cite{ultra}, Theorem 2, amounts to a statement 
``dual'', in a sense, to Theorem 1, where embeddings are replaced by surjections
(``dual'', and not dual in the strict sense, as the direct products are not 
replaced by the direct sums, and not all arrows are reversed). Theorem 1 was 
proved using what we have dubbed as the ``Robinson--Amitsur ultrafilter'', a way
to construct the desired ultrafilter borrowed from the old papers by S. Amitsur
and A. Robinson in the context of ring theory (as exemplified in the classical
textbook \cite{jacobson}, see proofs of Propositions 2.1 and 2.2 at p.~77); the
proof of Theorem 2 was an adaptation of the group- and set-theoretic arguments from the 
recent papers by G.~Bergman and N.~Nahlus (\cite{bergman-g} being one of them).
It was stated in \cite{ultra}, that though the statements of the two theorems 
sound very similar, the proofs are different and each proof cannot be adapted to
the ``dual'' situation. The present notes shows the last assertion to be wrong,
at least partially: we are giving a proof of a version of Theorem~1 using 
reasonings very similar to those used in the proof of Theorem 2. 

\subsection{Notation, definitions, conventions}\label{sec-def}

We assume basic knowledge of universal algebra and set theory, see 
\cite[\S 1]{ultra} for further details. All algebraic systems are of the same 
fixed (but arbitrary) signature. 

Our convention for writing composition of maps is from left to right, i.e., if
$X \overset{f}{\to} Y \overset{g}{\to} Z$, then 
$X \overset{f \circ g}{\longrightarrow} Z$ with $(f \circ g)(x) = g(f(x))$.

Given a family of algebraic systems $\{B_i\}_{i\in \mathbb I}$ indexed by a set
$\mathbb I$, and an ultrafilter $\mathscr U$ on $\mathbb I$, the symbols
$\prod_{i \in \mathbb I} B_i$ and $\prod_{\mathscr U} B_i$ denote the direct 
product, and the ultraproduct over $\mathscr U$, respectively.

Let $\mathbb J \subseteq \mathbb I$. Denote by
$$
p_{\mathbb I,\mathbb J}: \prod_{i\in \mathbb I} B_i \to 
\prod_{i\in \mathbb J} B_i 
$$
the canonical projection, defined for any 
$b\in \prod_{i\in \mathbb I} B_i$ by the formula
$$
(p_{\mathbb I,\mathbb J}(b))(i) = b(i) \text{ if } i\in \mathbb J .
$$

For any two sets $\mathbb I$, $\mathbb J$, and their disjoint union 
$\mathbb I \sqcup \mathbb J$, by
$$
d_{\mathbb I,\mathbb J}: 
(\prod_{i\in \mathbb I} B_i) \times (\prod_{i\in \mathbb J} B_i)
\simeq \prod_{i\in \mathbb I \sqcup \mathbb J} B_i
$$
we denote the canonical isomorphism.

For an algebraic system $A$, $\Delta_A: A \to A \times A$ (or just $\Delta$, if
there is no danger of confusion) denotes the diagonal map.

\subsection{Subdirect irreducibility, indecomposable embeddings}
\label{sec-decomp}

Let
\begin{equation}\label{eq-f}
f: A \hookrightarrow \prod_{i \in \mathbb I} B_i
\end{equation}
be an embedding of an algebraic system $A$ into the direct product of a family 
of algebraic systems. For such an embedding, we associate the set 
$\mathscr U(f)$ of subsets $\mathbb J \subseteq \mathbb I$ satisfying the 
following property: $f$ factors through the projection $p_{\mathbb I,\mathbb J}$, i.e., there exists an
embedding $f_{\mathbb J}: A \to \prod_{i\in \mathbb J} B_i$ such that the 
diagram
\begin{equation*}
\begin{tikzcd}[column sep=large,row sep=large]
A \arrow{rd}[swap]{f_{\mathbb J}} \arrow{r}{f} & 
                                    \prod\limits_{i\in \mathbb I} B_i 
                                    \arrow{d}{p_{\mathbb I,\mathbb J}} \\
                                  & \prod\limits_{i\in \mathbb J} B_i 
\end{tikzcd}
\end{equation*}
is commutative.

In what follows, we will repeatedly use the following simple

\begin{lemma}
For any embedding $f$ of the kind (\ref{eq-f}), $\mathscr U(f)$ is upward 
closed.
\end{lemma}

\begin{proof}
Let $\mathbb J \in \mathscr U(f)$, and 
$\mathbb J \subseteq \mathbb K \subseteq \mathbb I$. The condition 
$\mathbb J \in \mathscr U(f)$ is equivalent to that 
$f \circ p_{\mathbb I,\mathbb J}$ is an embedding, i.e., 
$\Ker(f \circ p_{\mathbb I,\mathbb J}) = \Delta(A)$. We have 
$$
\Delta(A) \subseteq \Ker(f \circ p_{\mathbb I,\mathbb K}) \subseteq 
\Ker(f \circ p_{\mathbb I,\mathbb K} \circ p_{\mathbb K,\mathbb J}) =
\Ker(f \circ p_{\mathbb I,\mathbb J}) = \Delta(A) ,
$$
and hence $\mathbb K \in \mathscr U(f)$, as required.
\end{proof}

Recall that given a cardinal $\kappa$, an algebraic system $A$ is called 
\emph{$\kappa$-subdirectly irreducible}, if for any embedding (\ref{eq-f}) of 
$A$ into the direct product of $< \kappa$ algebraic systems (i.e., with 
$\modulus{\mathbb I} < \kappa$), there is $i_0 \in \mathbb I$ such that 
$\{i_0\} \in \mathscr U(f)$. If $\kappa = \omega$ (i.e., only embeddings into 
finite direct products are considered), then $A$ is called 
\emph{finitely subdirectly irreducible}.

If $A$ is finitely subdirectly irreducible, then due to the canonical 
isomorphism $d_{\mathbb J,\mathbb I \backslash \mathbb J}$ which holds for any 
subset $\mathbb J \subseteq \mathbb I$, we have that either 
$\mathbb J \in \mathscr U(f)$, or 
$\mathbb I \backslash \mathbb J \in \mathscr U(f)$. In general, both these 
inclusions can hold simultaneously. We will call an embedding (\ref{eq-f})
\emph{decomposable}, if this is indeed the case, i.e., there is 
$\mathbb J \subset \mathbb I$, such that both $\mathbb J$ and 
$\mathbb I \backslash \mathbb J$ belong to $\mathscr U(f)$. Otherwise -- the
situation we are interested in -- if for any $\mathbb J \subseteq \mathbb I$, 
exactly one of the two inclusions $\mathbb J \in \mathscr U(f)$ and 
$\mathbb I \backslash \mathbb J \in \mathscr U(f)$ holds, an embedding 
(\ref{eq-f}) will be called \emph{indecomposable}.

Indecomposability is a strong condition, and far from every embedding of the 
kind (\ref{eq-f}) is indecomposable: a trivial example is the diagonal embedding
. Moreover, any embedding of the kind (\ref{eq-f}) 
may be ``doubled'' to a decomposable embedding in a trivial way, by taking the 
composition with the diagonal: 
$$
A \overset{\Delta}{\to} A \times A \overset{(f,f)}{\longrightarrow} 
(\prod_{i\in \mathbb I} B_i) \times (\prod_{i\in \mathbb I} B_i) .
$$

The following proposition shows that this is, essentially, how any decomposable
embedding can be obtained.

\begin{proposition}
A decomposable embedding $f: A \hookrightarrow \prod_{i \in \mathbb I} B_i$ can
be represented in the form $\Delta \circ g$ for some embedding 
$g: A \times A \hookrightarrow \prod_{i \in \mathbb I} B_i$.
\end{proposition}

\begin{proof}
By definition, there exists $\mathbb J \subset \mathbb I$ such that 
$\mathbb J, \mathbb I \backslash \mathbb J \in \mathscr U(f)$. 
Then 
$f = \Delta \circ (f_{\mathbb J}, f_{\mathbb I \backslash \mathbb J}) \circ 
     d_{\mathbb J,\mathbb I \backslash \mathbb J}$; thus we can put 
$g = (f_{\mathbb J}, f_{\mathbb I \backslash \mathbb J}) \circ 
     d_{\mathbb J,\mathbb I \backslash \mathbb J}$.
\end{proof}

In a somewhat opposite direction, under certain conditions an embedding of a 
finitely subdirectly irreducible algebraic system can be ``thinned'' to an 
indecomposable one.

\begin{proposition}
Let $A$ be non-trivial finitely subdirectly irreducible algebraic system, and 
$f: A \hookrightarrow \prod_{i\in \mathbb I} B_i$ an embedding such that the 
intersection of elements of any descending chain in $\mathscr U(f)$ lies in 
$\mathscr U(f)$. Then there is $\mathbb J \subseteq \mathbb I$ such that 
$f \circ p_{\mathbb I,\mathbb J}$ is an indecomposable embedding.
\end{proposition}

\begin{proof}
Consider the set 
$\mathcal S = \set{\mathbb J \subseteq \mathbb I}
{\mathbb J, \> \mathbb I \backslash \mathbb J \in \mathscr U(f)}$. If 
$\mathscr S$ is empty, then $f$ is indecomposable and we are done, so assume 
$\mathscr S$ is not empty. Let $\mathcal C$ be a descending chain of sets from 
$\mathcal S$. We have 
$\bigcap_{\mathbb J \in \mathcal C} \mathbb J \in \mathscr U(f)$. On the other
hand, 
$\mathbb I \backslash (\bigcap_{\mathbb J \in \mathcal C} \mathbb J) = 
 \bigcup_{\mathbb J \in \mathcal C} (\mathbb I \backslash \mathbb J) \in
 \mathscr U(f)$ due to the fact that $\mathscr U(f)$ is upward closed.
Therefore, any descending chain of elements of $\mathcal S$ has a lower 
bound (intersection of the elements from the chain), and by 
(the dual version of) the Zorn lemma, $\mathcal S$ has a minimal element 
$\mathbb J_0$. We claim that $\mathbb J_0$ is the required subset of 
$\mathbb I$.

Indeed, assume that there is $\mathbb J \subseteq \mathbb J_0$ such that
both $\mathbb J$ and $\mathbb J_0 \backslash \mathbb J$ belong to 
$\mathscr U(f \circ p_{\mathbb I,\mathbb J_0})$. We have
$$
f_{\mathbb J} = 
f \circ p_{\mathbb I,\mathbb J_0} \circ p_{\mathbb J_0,\mathbb J} =
f \circ p_{\mathbb I,\mathbb J}
$$
which shows that $\mathbb J \in \mathscr U(f)$. Similarly, 
$\mathbb J_0 \backslash \mathbb J \in \mathscr U(f)$. But then, since 
$\mathscr U(f)$ is upward closed, 
$\mathbb I \backslash \mathbb J \in \mathscr U(f)$. Hence 
$\mathbb J \in \mathscr S$, and by the minimality of $\mathbb J_0$, we have
$\mathbb J = \mathbb J_0$. But then 
$\varnothing \in \mathscr U(f \circ p_{\mathbb I,\mathbb J_0})$, a contradiction
with the non-triviality of $A$.
\end{proof}

\subsection{Embedding into an ultraproduct}\label{ss-main}

Our main result is the following

\begin{theorem}
Let $\kappa$ be an infinite cardinal, $A$ a $\kappa$-subdirectly irreducible 
algebraic system, and $f: A \hookrightarrow \prod_{i\in \mathbb I} B_i$ an indecomposable
embedding into the direct product of a family of algebraic systems 
$\{B_i\}_{i\in \mathbb I}$. Then there is a $\kappa$-complete ultrafilter 
$\mathscr U$ on the set $\mathbb I$ such that the composition of $f$ with the 
canonical homomorphism $\prod_{i\in \mathbb I} B_i \to \prod_{\mathscr U} B_i$,
is an embedding.
\end{theorem}

Before we proceed with the proof, let us discuss the differences with Theorem 1 
from \cite{ultra}. The latter theorem states a necessary and sufficient 
condition, but one direction, namely, the implication (ii) $\Rightarrow$ (i)
there, is trivial. The major difference is that here we do not require $\kappa$
to be strongly compact. On the other hand, we assume that $f$ is indecomposable.
Note that in the case $\kappa = \omega$, Theorem 1 from \cite{ultra} is 
stronger, as $\omega$ is strongly compact by the definition of ultrafilters and
the Zorn lemma.

The proof given here is similar to the proof of Theorem 2 from \cite{ultra} (see
also Theorem 2.10 from the recent paper \cite{alg-univ} which uses a similar 
argument, but in a more restrictive setting), and is modelled after the proofs 
of Lemma 7 and Proposition 8 from \cite{bergman-g}.

\begin{proof}
We will show that we can take $\mathscr U = \mathscr U(f)$. We have 
$\mathbb I \in \mathscr U(f)$, so $\mathscr U(f)$ is not empty; also, 
$\mathscr U(f)$ is upward closed. 

Suppose that two sets $\mathbb J$ and $\mathbb K$ belong to $\mathscr U(f)$.
Then $\mathbb I \backslash \mathbb K \notin \mathscr U(f)$, and, since 
$\mathscr U(f)$ is upward closed, we have
$\mathbb J \backslash \mathbb K \notin \mathscr U(f)$. Similarly, 
$\mathbb K \backslash \mathbb J \notin \mathscr U(f)$. Further, since, say,
$\mathbb J \in \mathscr U(f)$ and $\mathscr U(f)$ is upward closed, we obtain
$\mathbb J \cup \mathbb K \in \mathscr U(f)$, and since $f$ is indecomposable,
$\mathbb I \backslash (\mathbb J \cup \mathbb K) \notin \mathscr U(f)$. 

The set $\mathbb I$ can be decomposed to the disjoint union of four sets:
$$
\mathbb I = (\mathbb I \backslash (\mathbb J \cup \mathbb K)) \cup
            (\mathbb J \backslash \mathbb K) \cup
            (\mathbb K \backslash \mathbb J) \cup
            (\mathbb J \cap \mathbb K) .
$$
As we have just seen, the first three sets in this decomposition do not belong 
to $\mathscr U(f)$, and, since $A$ is (finitely) subdirectly irreducible, the 
fourth set, $\mathbb J \cap \mathbb K$, belongs to $\mathscr U(f)$.

Let us build the corresponding map $f_{\mathbb J \cap \mathbb K}$ explicitly. We
have: 
\begin{equation}\label{eq-eq}
f_{\mathbb J} \circ p_{\mathbb J,\mathbb J \cap \mathbb K} = 
f_{\mathbb K} \circ p_{\mathbb K,\mathbb J \cap \mathbb K} .
\end{equation}
Indeed,
$$
f_{\mathbb J} \circ p_{\mathbb J,\mathbb J \cap \mathbb K} = 
f \circ p_{\mathbb I,\mathbb J} \circ p_{\mathbb J,\mathbb J \cap \mathbb K} =
f \circ p_{\mathbb I,\mathbb J \cap \mathbb K} ,
$$
and similarly,
$$
f_{\mathbb K} \circ p_{\mathbb K,\mathbb J \cap \mathbb K} = 
f \circ p_{\mathbb I,\mathbb J \cap \mathbb K} .
$$
Define the map $f_{\mathbb J \cap \mathbb K}$ to be the map (\ref{eq-eq}). The 
computation above shows that 
$f \circ p_{\mathbb I,\mathbb J \cap \mathbb K} = f_{\mathbb J \cap \mathbb K}$.

Therefore, $\mathscr U(f)$ is closed upward and with respect to intersections,
and hence is a filter on $\mathbb I$. Since $f$ is indecomposable, 
$\mathscr U(f)$ is an ultrafilter.

To prove that $\mathscr U(f)$ is $\kappa$-complete, it is enough to show that 
for any decomposition $\mathbb I = \bigcup_{j \in \mathbf I} \mathbb I_j$ into 
the union of pairwise disjoint subsets $\{\mathbb I_j\}_{j \in \mathbf I}$, 
where $|\mathbf I| < \kappa$, at least one of $\mathbb I_j$'s belongs to 
$\mathscr U(f)$. But this follows directly from the $\kappa$-subdirect 
irreducibility of $A$: since 
$$
\prod_{i\in \mathbb I} B_i \simeq
\prod_{j \in \mathbf I}\Big(\prod_{i\in \mathbb I_j} B_i\Big) ,
$$
there is $j_0 \in \mathbf I$ such that $\mathbb I_{j_0} \in \mathscr U(f)$.

Factoring the embedding $f$ by the congruence 
$$
\theta= 
\set{(a,b) \in (\prod_{i\in \mathbb I} B_i) \times (\prod_{i\in \mathbb I} B_i)}
{\set{i\in \mathbb I}{a(i) = b(i)} \in \mathscr U(f)}
$$
defining the ultraproduct $\prod_{\mathscr U(f)} B_i$, we get the commutative 
diagram
$$
\begin{tikzcd}[column sep=large,row sep=large]
A  \arrow{d} \arrow{r}{f} & \prod\limits_{i\in \mathbb I} B_i \arrow{d} \\ 
A/(\theta \cap (A \times A)) \arrow{r} & \prod\limits_{\mathscr U(f)} B_i 
\end{tikzcd}
$$
where the vertical arrows are the canonical quotients by the respective 
congruences, and the bottom horizontal arrow is also an embedding. 

But if $(a,b) \in \theta \cap (A \times A)$, then $a,b\in A$, and, by the 
definition of $\mathscr U(f)$, $A$ is embedded into 
$\prod_{\set{i \in \mathbb I}{a(i) = b(i)}} B_i$. But then $a=b$, hence the 
congruence $\theta \cap (A \times A)$ coincides with the diagonal $\Delta(A)$, 
and $A$ is embedded into $\prod_{\mathscr U(f)} B_i$, as required.
\end{proof}

\subsection{Remark}\label{ss-remark}

The final remark concerns the ``classical'' case $\kappa = \omega$. Recall a 
corollary to this particular case of Theorem 1 from \cite{ultra}: if in a 
variety $\mathfrak A$ any free system is finitely subdirectly irreducible, then
an algebraic system $A \in \mathfrak A$ does not satisfy any nontrivial identity
within $\mathfrak A$ if and only if any free system in $\mathfrak A$ embeds into
an ultrapower of $A$. The corollary is obtained by a straightforward combination
of Theorem 1 with (the proof of) the Birkhoff theorem asserting an embedding of a free system in 
$\mathfrak A$ into a direct power of $A$.

This corollary could be interesting in the context of Universal
Algebraic Geometry (UAG for short), an attempt to generalize the ``classical'' 
algebraic geometry over polynomial rings to a geometry over arbitrary varieties
of algebraic systems. It was developed during the last two decades by B.~Plotkin
and his collaborators from one side, and by E.Yu.~Daniyarova, A.G.~Myasnikov and
V.N.~Remeslennikov from the other side (see, respectively, \cite{plotkin} and 
\cite{dmr}, and references therein). Indeed, the condition that a given 
algebraic system does not satisfy nontrivial identities within an ambient 
variety $\mathfrak A$, is equivalent to duality of certain categories -- of 
coordinate algebras and of algebraic sets -- defined in terms of free algebras 
of $\mathfrak A$ (\cite[Proposition 2.2]{plotkin}); this duality, in the 
parlance of UAG, means that syntax corresponds to semantics 
(\cite[Propositions 3.14 and 3.15]{plotkin}). On the other hand, the 
``unification theorem'' of UAG (\cite[Theorem 1]{dmr}) gives, under certain 
assumptions, equivalent conditions for embedding of an algebraic system into an
ultraproduct of another algebraic system.

\subsection*{Acknowledgements}

Thanks are due to Eugene Plotkin for useful remarks. The work was supported by
grant AP09259551 of the Ministry of Education and Science of the Republic of 
Kazakhstan.

\end{document}